\documentclass[11pt]{amsart}

\usepackage{AFBoix2015}

\begin{document}

\author[J.\,\`Alvarez Montaner]{Josep \`Alvarez Montaner}

\address{Departament de Matem\`atiques, Universitat Polit\`ecnica de Catalunya, Avinguda Diagonal 647,
Barcelona 08028, SPAIN} \email{Josep.Alvarez@upc.edu}

\author[A.\,F.\,Boix]{Alberto F.\,Boix}

\address{Department of Mathematics, Ben-Gurion University of the Negev, P.O.B. 653 Beer-Sheva 84105, Israel.}
\email{fernanal@post.bgu.ac.il}

\author[S.\,Zarzuela]{Santiago Zarzuela}

\address{Departament d'\`Algebra i Geometria, Universitat de Barcelona, Gran Via de les Corts Catalanes 585, Barcelona 08007, SPAIN} \email{szarzuela@ub.edu}

\thanks{The first author is partially supported by Generalitat de Catalunya 2014SGR-634 project
and Spanish Ministerio de Econom\'ia y Competitividad
MTM2015-69135-P. The second and third authors are  supported by
Spanish Ministerio de Econom\'ia y Competitividad MTM2016-7881-P.
The second author is partially supported by the Israel Science Foundation (grant No. 844/14).}

\title[Koszul complex over skew polynomial rings]{Koszul complex over skew polynomial rings}



\begin{abstract}
We construct a Koszul complex in the category of left skew polynomial rings associated to a flat
endomorphism that provides a finite free resolution of an ideal generated 
by a Koszul regular sequence. 


\end{abstract}

\dedicatory{Dedicated to Professor Gennady Lyubeznik on the occasion
     of his $60$th birthday}

\maketitle

\section{Introduction}
Let  $A$ be a commutative Noetherian ring of characteristic $p>0$ and $F:A \longrightarrow A$ the associated
Frobenius map for which $F(a)=a^p$ for all $a\in A$. The study of $A$-modules with an action of the Frobenius 
map has received a lot of attention over the last decades and is at the core of the celebrated 
theory of tight closure developed by M.\;Hochster and C.\;Huneke in \cite{HoHun1990} and the theory of $F$-modules 
introduced by G.\;Lyubeznik  \cite{Lyu1997}.

\vskip 2mm

To provide an action of the Frobenius on an $A$-module $M$ is equivalent to give a left $A[\Theta;F]$-module structure
on $M$. Here $A[\Theta;F]$ stands for the \emph{left skew polynomial ring} associated to $F$, which is an associative,
$\mathbb{N}$-graded, not necessarily commutative ring extension of $A$. More generally, we may also consider the 
skew polynomial rings $A[\Theta;F^e]$ associated to the $e$-th iterated Frobenius map and
the graded ring  $\mathcal{F}^M = \bigoplus_{e\geq 0} \mathcal{F}^M_e$ introduced by
G.\,Lyubeznik and K.\,E.\,Smith in \cite{LyubeznikSmith2001} that collects all the $A[\Theta;F^e]$-module structures
on $M$ or equivalently, all possible actions of $F^e$ on $M$. In the case that  $\mathcal{F}^M$ is principally generated 
then it is isomorphic to $A[\Theta;F]$. We want also to single out
here that, under the terminology \emph{skew polynomial ring of
Frobenius type}, Y.\,Yoshino \cite{Yoshino1994} studied $A[\Theta;F]$ and
left modules over it, and used this study to obtain some new results about
tight closure theory; following the same spirit, R.\,Y.\,Sharp 
\cite{Sharp2009} also studied the ideal theory of $A[\Theta;F].$ The
case when $A$ is a field was studied in detail by F.\,Enescu
\cite{Enescu2012}.



\vskip 2mm 

One may also develop a dual notion of  \emph{right skew polynomial ring} associated to $F$ that we denote 
$A[\varepsilon;F]$. A left $A[\varepsilon;F]$-module structure on $M$ is then equivalent to 
provide an action of a Cartier operator  as considered by M.\,Blickle and G.\,B\"ockle in \cite{BlickleBockle2011}.
The corresponding ring collecting all the $A[\varepsilon;F^e]$-module structures
on $M$, that we denote $\mathcal{C}^M = \bigoplus_{e\geq 0} \mathcal{C}^M_e$, was developed
by K.\,Schwede \cite{Schwede2011} and  M.\,Blickle \cite{Blickle2013} (see \cite{BlickleSchwede2013} for a nice survey)
and has been a hot topic in recent years because of its role  in the theory of \emph{test ideals}.

\vskip 2mm

Although the theory of skew polynomial rings is classical
(see \cite[Chapter 2]{GoodWar2004} and \cite[Chapter 1, Sections 2 and 6]{McRob2001}), 
the aim of this work is to go back to its basics and develop new tools  that should be potentially useful 
in the study of modules with a Frobenius or Cartier action.
Our approach will be in a slightly more general setting as 
we will consider a Noetherian commutative ring $A$ (not necessarily of positive characteristic) 
and a ring homomorphism $\varphi: A \longrightarrow A$. Although most of the results
are mild generalizations of the case of the 
Frobenius morphism we point out that this general approach has already been fruitful
(see \cite{SinghWalther2007}).

\vskip 2mm 


Now, we overview the contents of this paper for the
convenience of the reader. In Section \ref{preliminaries section} we introduce all the basics on left and right skew polynomial 
rings  associated to a morphism $\varphi$ that will be denoted by $A[\Theta;\varphi]$ and $A[\varepsilon;\varphi]$.  
More generally, given an $A$-module $M$, we will consider the corresponding rings $\mathcal{F}^{M,\varphi}$ and 
$\mathcal{C}^{M,\varphi}$. The main result of this Section is Theorem \ref{connectionCartierFrobenius} 
which states that, whenever $\varphi$ is a finite morphism, we have
an isomorphism between the graded pieces $\mathcal{F}^{E_R,\varphi}_e$ and $\mathcal{C}^{R,\varphi}_e$ 
induced by Matlis duality. Here, $E_R$ is the injective hull of the residue field of $R= A/I$ with 
$I \subseteq A:=\mathbb{K}[\![x_1,\ldots,x_n]\!]$ being an ideal.

\vskip 2mm

In Section \ref{our own Koszul complex} we present the main result of this work. Here we consider 
a flat morphism $\varphi$ satisfying some extra condition (which is naturally satisfied in the case of the Frobenius morphism). 
First we introduce the \emph{$\varphi$-Koszul complex} 
which is a Koszul-type complex associated to  $x_1, \dots , x_n \in A$  in the category of left $A[\Theta; \varphi]$-modules. 
In Theorem \ref{resolution} we prove that, whenever $x_1, \dots , x_n$ is an $A$-Koszul regular sequence (see Definition \ref{Koszul regular sequence}), 
the $\varphi$-Koszul complex  provides a free resolution of $A/I_n$ in the category of left 
$A[\Theta; \varphi]$-modules, where $I_n$ is the ideal generated by $x_1, \dots , x_n$. In the case where $A$ is a regular ring of positive characteristic $p>0$ and $\varphi=F^e$
is the $e$-th iteration of the Frobenius, we obtain a free resolution, in the category of left $A[\Theta; F^e]$-modules,
of the ideal $I_n$ (see Corollary \ref{resolution for the Frobenius}).

\vskip 2mm

To the best of our knowledge this is one of the first explicit examples 
of free resolutions of $A$-modules as modules over a skew polynomial ring. 
We hope that a development of this theory of free resolutions would provide, in the case of the Frobenius morphism, more insight 
in the study of modules with a Frobenius action.



\section{Preliminaries}\label{preliminaries section}

Let $A$ be a commutative Noetherian ring and  $\varphi : A \longrightarrow A$  a ring homomorphism.
Associated to $\varphi$ we may consider non necessarily commutative algebra extensions of $A$ that are
useful in the sense that, some non finitely generated $A$-modules become finitely generated
when viewed over these extensions. In this section we will collect the basic facts on this theory
keeping always an eye on the case of the Frobenius morphism.

\vskip 2mm

First we recall that $\varphi$ allows to describe  a covariant functor, the \emph{pushforward} or \emph{restriction of scalars} functor $\varphi_*$, from the category of left $A$-modules to the category of left $A$-modules. Namely, given  an $A$-module $M$,  $\varphi_* M$ 
is the $(A,A)$-bimodule having the usual $A$-module structure on the right but the left structure is twisted by  $\varphi$; that is, 
if we denote by $\varphi_* m$ ($m\in M$) an arbitrary element of $\varphi_* M$ then, for any $a\in A$
\[
a\cdot\varphi_* m:=\varphi_* (\varphi(a)m).
\]
Moreover, given a map $g: M \longrightarrow N$ of left $A$-modules we can define a map 
$\varphi_*g: \varphi_*M \longrightarrow \varphi_*N$ 
of left $A$-modules by setting, for any $m\in M$,
\[
\varphi_* g (\varphi_* m):=\varphi_* g(m).
\]

\subsection{Left and right skew-polynomial rings}
The most basic ring extensions given by a ring homomorphism are the  so-called \emph{skew-polynomial rings} or 
\emph{Ore extensions}. This is a classical object of study, and we refer either to \cite[Chapter 2]{GoodWar2004}
or \cite[Chapter 1, Sections 2 and 6]{McRob2001} for further details.

\begin{df} \label{df:Skew}
Let $A$ be a commutative Noetherian ring and $\varphi: A \longrightarrow A$  
a ring homomorphism.
\vskip 2mm 

$\bullet$ \emph{Left skew polynomial ring} of $A$ \emph{determined by $\varphi$}: 
\[
A[\Theta;\varphi]:=\frac{A\langle\Theta\rangle}{\langle\Theta a -\varphi (a)\Theta \mid\quad a\in A\rangle},
\]
that is,  $A[\Theta;\varphi]$ is a free left $A$-module with  basis 
$\{\theta^i\}_{i\in{\mathbb N}_0}$ and $\theta\cdot a=\varphi (a)\theta$ for all $a\in A$.
\vskip 2mm 

$\bullet$ \emph{Right skew polynomial ring} of $A$ \emph{determined by $\varphi$}: 
\[
A[\varepsilon;\varphi]:= \frac{A\langle\varepsilon\rangle}{\langle a\varepsilon - \varepsilon\varphi(a) \mid\quad a\in A\rangle},
\]
that is,  $A[\varepsilon;\varphi]$ is a free right $A$-module with  basis  $\{\varepsilon^i\}_{i\in{\mathbb N}_0}$ and $a\cdot\varepsilon=\varepsilon\varphi (a)$ for all $a\in A$.
\end{df}

\begin{rk} \label{opposite}
The reader will easily note that $A[\Theta;\varphi]^{op}\cong A[\varepsilon;\varphi]$ and
$A[\varepsilon;\varphi]^{op}\cong A[\Theta;\varphi],$ where $(-)^{op}$ denotes the opposite ring. 
\end{rk}

Both ring extensions are determined by the following \emph{universal property}.

\begin{prop}
Let $B=A[\Theta;\varphi]$ be the left   (resp. $B=A[\varepsilon;\varphi]$ be the right)  skew polynomial ring of $A$ determined 
by $\varphi$. Suppose that we have a ring $T$, a ring homomorphism $\phi: A \longrightarrow T$
 and an element $y\in T$ such that, for each $a\in A$, $$y\phi (a)=\phi(\varphi (a))y \qquad ({\rm resp.} \hskip 2mm \phi (a)y=y\phi(\varphi (a))).$$Then, there is a unique ring homomorphism $\xymatrix@1{B\ar[r]^-{\psi}& T}$ such that $\psi (\Theta)=y$ (resp. $\psi (\varepsilon)=y$) which makes the triangle
\begin{equation*}
\xymatrix{A\ar[dr]^-{\phi}\ar@{^{(}->}[rr]& \hbox{}& B\ar@{-->}[dl]^-{\psi}\\
\hbox{}& T& \hbox{}} 
\end{equation*}
commutative.
\end{prop}

This universal property allow us to provide more general examples of 
skew polynomial rings.

\begin{ex}
Let $u\in A$. As $u\Theta$ is an element of $A[\Theta ;\varphi]$ such that, for any $a\in A$,
\[
(u\Theta)a=u\Theta a=u\varphi (a)\Theta =\varphi (a) (u\Theta),
\]
it follows from the universal property for left skew polynomial rings that there is a unique $A$--algebra homomorphism $\xymatrix@1{A[\Theta '; \varphi]\ar[r]& A[\Theta ;\varphi]}$ sending $\Theta '$ to $u\Theta;$ we shall denote the image of this map by $A[u\Theta ;\varphi]$. The previous argument shows that $A[u\Theta ;\varphi]$ can be also regarded as a left skew polynomial ring. Analogously, 
$A[\varepsilon u ;\varphi]$ can be also regarded as right skew polynomial ring since we have, for any $a\in A$,
\[
a(\varepsilon u)=\varepsilon\varphi (a)u =(\varepsilon u)\varphi (a).
\]
\end{ex}

\subsubsection{The case of Frobenius homomorphism}
We single out the  case where  $A$ is a ring of positive characteristic $p>0$ and  $\varphi=F$ is the Frobenius
homomorphism. The specialization of Definition \ref{df:Skew} to this case is the following: 
 
  \vskip 2mm
 
$\bullet$ The \emph{Frobenius skew polynomial ring} of $A$ is the non-commutative graded ring $A[\Theta ;F]$; that is, the free left $A$-module with basis $\{\Theta^e\}_{e\in\N_0}$ and right multiplication given by
\[
\Theta\cdot a=a^p \Theta .
\]

\vskip 2mm

$\bullet$ The \emph{Cartier skew polynomial ring} of $A$ is the non-commutative graded ring $A[\varepsilon ;F]$; that is, the free right $A$-module with basis $\{\varepsilon^e\}_{e\in\N_0}$ and left multiplication given by
\[
a\cdot\varepsilon:=\varepsilon a^p.
\]

It turns out that  $A[\Theta; F]$  is rarely left or right Noetherian, as it is explicitly described 
 by Y.\,Yoshino in \cite[Theorem (1.3)]{Yoshino1994}. 
He also provided effective descriptions of basic examples of left modules over $A[\Theta ;F]$. Namely, the ring $A$, the localizatons
$A_a$ at elements $a\in A$ and local cohomology modules $H_I^i (A)$, where $I$ is any ideal of $A$, have an abstract structure as finitely generated left $A[\Theta ;F]$-modules that we collect in the following result. We refer to \cite[pp. 2490--2491]{Yoshino1994} for further details.

\begin{prop} \label{localizacion con Frobenius}
Let $A$ be a commutative Noetherian ring of characteristic $p$, let $a\in A$ be any element and let $J(a)$ denote the left ideal of $A[\Theta ;F]$ generated by the infinite set
\[
\{a^{p^e-1}\Theta^e -1\mid\quad e\in\N\}.
\]
Then, the following statements hold.
\begin{enumerate}[(i)]

\item $A$ has a natural structure as left $A[\Theta ;F]$-module given by
\[
A\cong A[\Theta;F]/A[\Theta;F]\langle\Theta -1\rangle \cong  A[\Theta; F]/ J(1).
\]

\end{enumerate}
Moreover, if $A$ is a local ring:
\begin{enumerate}[(i)]
\item[(ii)] The localization $A_a$ has a natural structure as left $A[\Theta;F]$-module given by

\[
A_a \cong  A[\Theta; F]/ J(a).
\]

\item[(iii)] The \v{C}ech complex of $A$ with respect to $a_1,\ldots ,a_t$
\[
0\longrightarrow A\longrightarrow\bigoplus_{i=1}^t A_{a_i}\longrightarrow\bigoplus_{1\leq i<j\leq t} A_{a_i a_j}\longrightarrow\ldots\longrightarrow\bigoplus_{i=1}^t A_{a_1\cdots\widehat{a_i}\cdots a_t}\longrightarrow A_{a_1\cdots a_t}\longrightarrow 0
\]
is a complex of left $A[\Theta ;F]$-modules which is isomorphic to the following complex:
\begin{align*}
& 0\longrightarrow A[\Theta ;F]/J(1)\longrightarrow\bigoplus_{i=1}^t A[\Theta ;F]/J(a_i)\longrightarrow\bigoplus_{1\leq i<j\leq t} A[\Theta;F]/J(a_i a_j)\longrightarrow\ldots\\
& \ldots\longrightarrow\bigoplus_{i=1}^t A[\Theta ;F]/J(a_1\cdots\widehat{a_i}\cdots a_t)\longrightarrow A[\Theta ;F]/J(a_1\cdots a_t)\longrightarrow 0.
\end{align*}

\item[(iv)] Any local cohomology module $H_I^i (A)$ has an abstract structure as a finitely generated left $A[\Theta;F]$-module.
For example, if $(A,\mathfrak{m},\K)$ is a local ring of characteristic $p$ of dimension $d\geq 1$, and  $a_1,\ldots ,a_d$ is a system of parameters for $A,$ then
\[
H_{\mathfrak{m}}^d (A)\cong A[\Theta ;F]/\left(J(a_1\cdots a_d)+A[\Theta; F]\langle a_1,\ldots ,a_d\rangle\right).
\]
\end{enumerate}

\end{prop}

\vskip 2mm

In this section we are going to give a natural generalization of skew polynomial rings associated to a ring homomorphism $\varphi$.
Before doing so, we briefly recall how to give a module structure over a skew polynomial ring by means of the so-called
$\varphi$ and $\varphi^{-1}$-linear maps.

\vskip 2mm

\begin{df}
Let $A$ be a commutative Noetherian ring,  $\varphi: A \longrightarrow A$ a ring homomorphism and 
 $M$ be an $A$-module. Given $\psi, \phi \in\End_A (M):$

\begin{enumerate}[(i)]

\item We say that $\psi$ is $\varphi$-\emph{linear}\index{map!$\varphi$-linear} provided $\psi (am)=\varphi (a)\psi (m)$ for any $(a,m)\in A\times M$.

\item We say that $\phi$ is $\varphi^{-1}$-\emph{linear}\index{map!$\varphi^{-1}$-linear} provided $\phi (\varphi(a)m)=a\phi (m)$ for any $(a,m)\in A\times M$.
\end{enumerate}

\end{df}
\vskip 2mm

We denote by $\End_{\varphi} (M)$ and  $\End_{\varphi^{-1}} (M)$ the $A$-endomorphisms of $M$ which are $\varphi$-linear and
 $\varphi^{-1}$-linear respectively. These endomorphisms can be interpreted in terms of the pushforward functor since we have the following bijections:
\begin{align*}
& \xymatrix{\End_{\varphi} (M)\ar[r] & \Hom_A (M,\varphi_* M)} ,  & \hskip 2mm & \End_{\varphi^{-1}} (M)\longrightarrow\Hom_A (\varphi_* M,M) \\
& \hskip 5mm \psi\longmapsto [m\longmapsto\varphi_* (\psi (m))] &   \hskip 2mm &   \hskip 5mm \psi\longmapsto [\varphi_* m\longmapsto\psi (m)]
\end{align*}

\vskip 2mm

An $A[\Theta ;\varphi]$-module is simply an $A$-module $M$ together with a suitable action of $\Theta$ on $M$. Actually,
one only needs to consider a $\varphi$-linear  map $\psi: M \longrightarrow M$ and define $\Theta \cdot m := \psi(m)$ for all
$m\in M$. 
Analogously, an $A[\varepsilon ;\varphi]$-module is an $A$-module $M$ together
with an action of $\varepsilon$ given by
a $\varphi^{-1}$-linear map $\phi: M \longrightarrow M;$
we record all these simple remarks into the following: 

\begin{prop} Let $M$ be an $A$-module. 

\begin{enumerate}[(i)]

\item There is a bijective correspondence between $\End_{\varphi} (M)$ and the left $A[\Theta; \varphi]$-module structures which can be attached to $M$.\index{ring!skew polynomial|)}

\vskip 2mm

\item There is a bijective correspondence between $\End_{\varphi^{-1}} (M)$ and the left $A[\varepsilon; \varphi]$-module structures which can be attached to $M$.\index{ring!skew polynomial|)}

\end{enumerate}

\end{prop}

\vskip 2mm

More generally, we may consider the $e$-th powers $\varphi^e$ of the ring homomorphism $\varphi$
and define the corresponding notion of  $\varphi^e$ and $\varphi^{-e}$-linear maps. 
We may collect all these morphisms in a suitable algebra that would provide a generalization 
of the  Frobenius algebra introduced by G.\,Lyubeznik and K.\,E.\,Smith in \cite[Definition 3.5]{LyubeznikSmith2001}
and the Cartier algebra considered by K.\, Schwede \cite{Schwede2011}  and generalized by M.\, Blickle \cite{Blickle2013}
(see also \cite{BlickleSchwede2013}).

\begin{df} 
Let $A$ be a commutative Noetherian ring and $\varphi: A \longrightarrow A$ a ring homomorphism.
Let $M$ be an $A$-module. Then we define:
\vskip 2mm 

$\bullet$ \emph{Ring of $\varphi$-linear operators} on $M$: Is the associative, $\mathbb{N}$-graded, not necessarily commutative ring
\[
\cF^{M ,\varphi} :=\bigoplus_{e\geq 0}\cF_e^{M, \varphi},
\]
where $\cF_e^{M, \varphi}:= \Hom_A (M, \varphi_*^e M)= \End_{\varphi^{e}} (M).$ 
A product is defined as
$$
\psi_{e'}\cdot\psi_e:=\varphi_*^e \left(\psi_{e'}\right)\circ\psi_e\in\cF_{e+e'}^{M,\varphi}.
$$
for any given $\psi_e\in\cF_e^{M, \varphi}$ and $\psi_{e'}\in\cF_{e'}^{M , \varphi}$.

\vskip 2mm 

$\bullet$ \emph{Ring of $\varphi^{-1}$-linear operators} on $M$: Is the associative, $\mathbb{N}$-graded, not necessarily commutative ring
\[
\cC^{M,\varphi}:= \bigoplus_{e\geq 0} \cC_e^{M, \varphi},   
\]
where $\cC_e^{M, \varphi}:= \Hom_A (\varphi_*^e M,M)= \End_{\varphi^{-e}} (M).$

A product is defined as
$$
\phi_{e'}\cdot\phi_e:=\phi_{e'}\circ \varphi_*^{e'}(\phi_e) \in \cC_{e+e'}^{M,\varphi}
$$
for any given $\phi_e\in\cC_e^{M, \varphi}$ and $\phi_{e'}\in\cC_{e'}^{M, \varphi}$.

\end{df}

\vskip 2mm

Actually, the generalization of Cartier algebra given by M.\, Blickle in \cite{Blickle2013} would be interpreted in our context as follows:

\begin{df}
An $A$-\emph{Cartier algebra with respect to $\varphi$} is an $\N$-graded $A$-algebra
\[
\cC^{\varphi}:=\bigoplus_{e\geq 0}\cC_e^{\varphi}
\]
such that, for any $(a,\phi_e)\in A\times\cC_e^{\varphi}$, we have that $a\cdot\phi_e =\phi_e\cdot\varphi^e (a)$. The $A$-algebra structure of $\cC^{\varphi}$ is given by the natural map from $A$ to $\cC_0^{\varphi}$. 
We also assume that the structural map $\xymatrix@1{A\ar[r]& \cC_0^{\varphi}}$ is surjective.
\end{df}

\vskip 2mm

We point out that $\cC^{M,\varphi}$ is generally NOT an $A$-Cartier algebra with respect to $\varphi$, since $\cC_0^{M, \varphi}=\End_A (M)$ and therefore the natural map $\xymatrix@1{A\ar[r]& \End_A (M)}$ is, in general, not surjective. Nevertheless, if $M=A/I$ (where $I$ is any ideal of $A$) then $\End_A (M)=A/I$ and therefore it follows that $\cC^{A/I, \varphi}$ is an $A$-Cartier algebra.

\vskip 2mm

Whenever the ring of $\varphi$-linear  (resp. $\varphi^{-1}$-linear)  operators
is principally generated it is isomorphic to a left (resp. right) skew polynomial ring. This is the case of the ring of 
$\varphi$-linear operators of the ring $A;$ notice that, when $A$ is of prime characteristic
and $\varphi=F$ is the Frobenius map, this was already observed by
G.\,Lyubeznik and K.\,E.\,Smith \cite[Example 3.6]{LyubeznikSmith2001}.

\begin{ex}
We have that $\cF^{A, \varphi}\cong A[\Theta ; \varphi]$. Indeed, fix $e\in\N$ and let $\psi_e\in\cF_e^{A, \varphi}$. We point out that, for any $a\in A$,
\[
\psi_e (a)=\psi_e (a\cdot 1)=\varphi^e (a)\psi_e (1)=\psi_e (1) \varphi^e (a).
\]
In this way, set
\begin{align*}
& \xymatrix{\cF_e^{A, \varphi} \ar[r]^-{b_e}& A\Theta^e}\\
& \psi_e\longmapsto\psi_e (1) \Theta^e .
\end{align*}
The previous straightforward calculation shows the injectivity of this map. In fact, it is a bijective map with inverse
\begin{align*}
& \xymatrix{A\Theta^e\ar[r]& \cF_e^{A, \varphi}}\\
& a\Theta^e\longmapsto a\varphi^e.
\end{align*}
In this way, setting $\xymatrix@1{\cF^{A, \varphi}\ar[r]^-{b}& A[\Theta ; \varphi]}$ as the unique map of rings given in degree $e$ by $b_e$ it follows that $b$ is an isomorphism of graded algebras.
\end{ex}

\subsubsection{Duality between Cartier algebras and Frobenius algebras}

 Let $A=\mathbb{K}[\![x_1,\ldots,x_d]\!]$ be a formal power series ring with $d$ indeterminates over a  field $\mathbb{K}$,
  $I \subseteq A$ an ideal and $R:=A/I$. 
 Let $\varphi : A \longrightarrow A$ be a ring homomorphism such that $\varphi_*^e A$ is finitely generated as $A$-module.
The aim of this subsection is to establish an explicit  correspondence, at the level of graded pieces, 
between $\cC^{R,\varphi}$ and $\cF^{E_R, \varphi}$ given by Matlis duality. This would extend 
the correspondence given in the case of the Frobenius map  (cf.\,\cite[Proposition 5.2]{BlickleBockle2011} and \cite[Theorem 1.20 and
Corollary 1.21]{SharpYoshino2011}). 

\begin{teo}\label{connectionCartierFrobenius}
Let $\varphi : A \longrightarrow A$ be a ring homomorphism such that $\varphi_*^e A$ is finitely generated as $A$-module. Then we have that
\[
\Hom_A (\varphi_*^e R,R)^{\vee}\cong\Hom_A (E_R,\varphi_*^e E_R)\quad\text{and}\quad\Hom_A(E_R, \varphi_*^e E_R)^{\vee}\cong\Hom_A (\varphi_*^e R,R).
\]

\end{teo}
Before proving this theorem we have to show a previous statement which we shall need during its proof; albeit the below result was obtained by F.\,Enescu and M.\,Hochster in \cite[Discussion (3.4)]{EnescuHochster2008} (see also \cite[Lemma (3.6)]{Yoshino1994}), we review here their proof for the convenience of the reader.

\begin{lm}\label{envolvente inyectiva y cambio de base}
Let $\xymatrix@1{(A,\mathfrak{m},\mathbb{K})\ar[r]& (B,\mathfrak{n},\mathbb{L})}$ be a local homomorphism of local rings, and suppose that $\mathfrak{m}B$ is $\mathfrak{n}$-primary and that $\mathbb{L}$ is finite algebraic over $\mathbb{K}$ (both these conditions hold if $B$ is module-finite over $A$). Let $E:=E_A (\mathbb{K})$ and $E_B(\mathbb{L})$ denote choices of injective hulls for $\mathbb{K}$ over $A$ and for $\mathbb{L}$ over $B$, respectively. Then, the functor $\Hom_A (-,E)$, on $B$-modules, is isomorphic with the functor $\Hom_B (-,E_B(\mathbb{L}))$.
\end{lm}

\begin{proof}
First of all, we underline that $\Hom_A (-,E)$, on $B$-modules, can be identified via adjunction with
\[
\Hom_A \left((-)\otimes_A B, E\right)\cong\Hom_B (-,\Hom_A (B,E))
\]
and therefore $\Hom_A (B,E)$ is injective as $B$-module. Moreover, as $\mathfrak{m}B$ is $\mathfrak{n}$-primary any element of $\Hom_A (B,E)$ is killed by a power of $\mathfrak{n}$ and therefore
\[
\Hom_A (B, E)\cong E_B (\mathbb{L})^{\oplus l}.
\]
In this way, it only remains to check that $l=1$. Indeed, we note that
\[
\Hom_A (\mathbb{L},E)\cong\Hom_A (\mathbb{L},\mathbb{K}).
\]
However, as $A$-module, $\Hom_A (\mathbb{L},\mathbb{K})$ is abstractly isomorphic to $\mathbb{L}$ (here we are using the assumption that $\mathbb{L}$ is finite algebraic over $\mathbb{K}$). Thus, all these foregoing facts imply that
\[
E_B (\mathbb{L})\cong\Hom_A (B,E),
\]
hence $\Hom_A \left((-)\otimes_A B, E\right)\cong\Hom_B (-,E_B (\mathbb{L}))$ and we get the desired conclusion.
\end{proof}

\begin{proof}[Proof of Theorem \ref{connectionCartierFrobenius}]
First of all, we underline that
\[
\Hom_A (\varphi_*^e R,R)^{\vee}\cong\Hom_A (E_R,\varphi_*^e (R)^{\vee}).
\]
Now, let $E_*$ be the injective hull of the residue field of $\varphi_*^e R$. In this way, from Lemma \ref{envolvente inyectiva y cambio de base} we deduce that $\Hom_A (-,E)\cong\Hom_{\varphi_*^e A}(-,E_*)$ as functors of $\varphi_*^e A$-modules. Therefore, combining all these facts joint with the exactness of $\varphi_*^e$ it follows that
\[
\varphi_*^e(R)^{\vee}\cong\Hom_A (\varphi_*^eR,E)\cong\Hom_{\varphi_*^eA}(\varphi_*^e R,E_*)\cong \varphi_*^e\Hom_A(R,E)\cong\varphi_*^e E_R.
\]
Thus, taking into account this last chain of isomorphisms one obtains the first desired conclusion.

On the other hand, using once more Lemma \ref{envolvente inyectiva y cambio de base} it turns out that
\[
\varphi_*^e (E_R)^{\vee}\cong\Hom_A (\varphi_*^e E_R, E)\cong\Hom_{\varphi_*^e A} (\varphi_*^e E_R, \varphi_*^e E)\cong \varphi_*^e (E_R^{\vee})\cong \varphi_*^e R.
\]
Thus, bearing in mind this last chain of isomorphisms it follows that
\[
\Hom_A (E_R, \varphi_*^e E_R)^{\vee}\cong\Hom_A (\varphi_*^e(E_R)^{\vee},E_R^{\vee})\cong\Hom_A (\varphi_*^e R,R),
\]
just what we finally wanted to show.
\end{proof}

\subsubsection{The case of Frobenius homomorphism}

Once again we single out the  case where  $A$ is a ring of positive characteristic $p>0$ and  $\varphi=F$ is the Frobenius
homomorphism. In this case we adopt the terminology of $p^e$ and $p^{-e}$-linear maps or, following \cite{Anderson2000}, 
\emph{Frobenius } and \emph{Cartier} linear maps.

\begin{df}
Let $M$ be an $A$-module and $\psi, \phi \in\End_A (M)$.

\begin{enumerate}[(i)]

\item We say that $\psi$ is $p^e$-\emph{linear} provided $\psi (am)=a^{p^e} \psi (m)$ for any $(a,m)\in A\times M$. Equivalently, $\psi\in\Hom_A (M,F_*^e M)$.

\vskip 2mm

\item We say that $\phi$ is $p^{-e}$-\emph{linear} provided $\phi (a^{p^e}m)=a \phi (m)$ for any $(a,m)\in A\times M$. Equivalently, $\phi\in\Hom_A (F_*^e M,M)$.

\end{enumerate}
\end{df}

\vskip 2mm

The corresponding rings of $p^e$ and $p^{-e}$-linear maps are defined as follows:

\vskip 2mm

\begin{df}\label{definicion oficial de algebra de Cartier}
Let $A$ be a commutative Noetherian ring of prime characteristic $p$ and let $M$ be an $A$-module.

\begin{enumerate}[(i)]

\item The \emph{Frobenius algebra} attached to $M$ is the associative, $\N$-graded, not necessarily commutative ring
\[
\cF^M:=\bigoplus_{e\geq 0} \Hom_A \left(M,F_*^e M\right).
\]

\item The \emph{Cartier algebra} attached to $M$ is the associative, $\N$-graded, not necessarily commutative ring
\[
\cC^M:=\bigoplus_{e\geq 0} \Hom_A \left(F_*^e M, M\right).
\]

\end{enumerate}

\end{df}

In this work we are mainly interested in the case where the Frobenius  (resp. Cartier) algebra of a module is principally generated and thus
isomorphic to the left (resp. right) skew polynomial ring. G.\,Lyubeznik and K.\,E.\,Smith 
already carried out such an example in \cite[Example 3.7]{LyubeznikSmith2001}.
\begin{ex}\label{computation of Frobenius algebra of top local cohomology module}
Let $(A,\mathfrak{m},\mathbb{K})$ be a local ring of characteristic $p$. Then
\[
\cF^{H_{\mathfrak{m}}^{\dim (A)}(A)}\cong S[\Theta ;F],
\]
where $S$ denotes the $S_2$-ification of the completion $\widehat{A}$. 
\end{ex}

Another source of examples is given by the following result.
\begin{prop}\label{criterio de principal generacion}
Let $(A,\mathfrak{m},\mathbb{K})$ be a complete $F$-finite local ring of characteristic $p$ and $E_A$ will stand for a choice of injective hull of $\K$ over $A$. Then, the following statements hold.
\begin{enumerate}[(i)]

\item If $A$ is quasi Gorenstein then $\cF^{E_A}$ is principal.

\item If $A$ is normal then $\cF^{E_A}$ is principal if and only if $A$ is Gorenstein.

\item If $A$ is a $\Q$-Gorenstein normal domain then $\cF^{E_A}$ is a finitely generated $A$-algebra if and only if $p$ is relatively prime with the index of $A$.

\item  If $A$ is a $\Q$-Gorenstein normal domain then $\cF^{E_A}$ is principal if and only if the index of $A$ divides $p-1$.

\end{enumerate}
\end{prop}

\begin{proof}
If $A$ is quasi Gorenstein then $E_A\cong H_{\mathfrak{m}}^{\dim (A)}(A)$. But we have seen in Example \ref{computation of Frobenius algebra of top local cohomology module} that, under our assumptions, $\cF^{H_{\mathfrak{m}}^{\dim (A)}(A)}\cong A[\Theta ;F]$; indeed, $A$ is complete and any quasi Gorenstein ring is, in particular, $S_2$. The second part is proved in \cite[Example 2.7]{Blickle2013}. On the other hand, part (iii) follows combining \cite[Proposition 4.3]{KatzmanSchwedeSinghZhang2013} and \cite[Theorem 4.5]{EnescuYao2014}; finally, part (iv) also follows from \cite[Proposition 4.3]{KatzmanSchwedeSinghZhang2013}.
\end{proof}

We point out that an explicit description of   $\cF^{E_A}$ can be obtained using the following result due to R.\,Fedder (cf.\,\cite[pp.\,465]{Fedder1983}).

\begin{teo}\label{el teorema de Fedder}
 Let $A=\K [\![x_1, \ldots ,x_d]\!]$ be a formal power series ring over a field $\K$ of prime characteristic $p$. Let $I$ be an arbitrary ideal of $A$. Then, one has that
\[
\cF^{E_R}\cong\bigoplus_{e\geq 0}\left\{(I^{[p^e]}:_A I)/I^{[p^e]}\right\}\Theta^e ,
\]
where $E$ denotes a choice of injective hull of $\K$ over $A$, $R:=A/I$, $\Theta$ is the standard Frobenius action on $E$ and $E_R:=(0:_E I)$.
\end{teo}

In \cite{AlvarezBoixZarzuela2012} we used this result to study Frobenius  algebras associated to Stanley-Reisner rings.
It turns out that, whenever they are principally generated, they are isomorphic to $A[u\Theta ;F]$ with  $u=x_1^{p-1}\cdots x_d^{p-1}$.

\section{The $\varphi$-Koszul chain complex}\label{our own Koszul complex}
Let  $S=\K [x_1,\ldots ,x_n]$ be the polynomial ring with coefficients on a commutative ring $\K,$ and
let $\varphi: S \longrightarrow S$ be a flat map of $\K$--algebras
satisfying the extra condition that for any $1\leq i\leq n$, $\varphi (x_i)\in\langle x_i\rangle$. Thus, 
there are non--zero elements $s_1,\ldots, s_n$ of $S$ such that $\varphi (x_i)=s_i x_i,$  for each $1\leq i\leq n$.

\vskip 2mm

\begin{rk}
If $\K$ is a field of positive characteristic $p>0$  and $\varphi=F^e$ is the iterated Frobenius morphism,
then this extra condition is naturally satisfied. Indeed, $F^e(x_i)= x_i^{p^e}$ so $s_i= x_i^{p^e-1}$.
More generally \cite[Example 2.2]{SinghWalther2007}, if $\K$ is any field and $t\geq 1$ is an integer, then the
$\K$--linear map $\varphi$ on $S$ sending each $x_i$ to $x_i^t$ also satisfies
this condition; indeed, in this case, $s_i=x_i^{t-1}.$
\end{rk}

\vskip 2mm

Our aim is to construct a Koszul complex in the category of $S[\Theta; \varphi]$ -modules  that we will denote as $\varphi$-Koszul complex.
To begin with, let us fix some notations. 
Let 
\[
K(x_1,\dots,x_n):= \xymatrix{0\ar[r]& K_n\ar[r]^-{d_n}& K_{n-1}\ar[r]& \ldots \ar[r]& K_2\ar[r]^-{d_2}\ar[r]& K_1\ar[r]^-{d_1}& K_0\ar[r]& 0}
\]
be the Koszul chain complex of $S$ with respect to $x_1,\ldots ,x_n$ (regarded as a chain complex in the category of left $S$-modules) and suppose that each differential $d_l$ is represented by right multiplication by matrix $M_l$.  
Moreover, for each $l\geq 0$ $M_l^{[\varphi]}$ denotes the matrix obtained by applying to each entry of $M_l$ the map 
$\varphi$. In particular,  for each entry $m,$ the sign of $m$ is equal to the sign of $\varphi (m).$

\begin{df}\label{el complejo de Frobenius-Koszul}
We define the \emph{$\varphi$-Koszul chain complex} with respect to $x_1,\ldots ,x_n$ as the chain complex
\[
\FK_{\bullet}(x_1,\dots,x_n):=\xymatrix{0\ar[r]& \FK_{n+1}\ar[r]^-{\partial_{n+1}}& \FK_n\ar[r]^-{\partial_n}& \ldots\ar[r]^-{\partial_1}& \FK_0\ar[r]&  0.}
\]
Here, for each $0\leq l\leq n+1$,
\begin{align*}
\FK_l:=& \bigoplus_{1\leq i_1<\ldots<i_l\leq n}S[\Theta;\varphi](\mathbf{e}_{i_1}\wedge\ldots\wedge\mathbf{e}_{i_l}) \hskip 2mm \oplus\bigoplus_{1\leq j_1<\ldots<j_{l-1}\leq n}S[\Theta;\varphi](\mathbf{e}_{j_1}\wedge\ldots\wedge\mathbf{e}_{j_{l-1}}\wedge u),
\end{align*}
where $\mathbf{e}_1,\ldots ,\mathbf{e}_n$ corresponds respectively to $x_1,\ldots ,x_n$ and $u$ corresponds to $\Theta-1$. Here, we are adopting the convention that $\FK_0:=S[\Theta;\varphi]$.

Moreover, one defines $\xymatrix@1{\FK_l\ar[r]^-{\partial_l}& \FK_{l-1}}$ as the unique homomorphism of left $S[\Theta;\varphi]$-modules which, on basic elements, acts in the following manner: 

\vskip 2mm

$\cdot$ Given $1\leq i_1<\ldots<i_l\leq n$, set
\[
\partial_l \left(\mathbf{e}_{i_1}\wedge\ldots\wedge\mathbf{e}_{i_l}\right):=\sum_{r=1}^l (-1)^{r-1}x_{i_r}\left(\mathbf{e}_{i_1}\wedge\ldots\wedge\mathbf{e}_{i_{r-1}}\wedge\mathbf{e}_{i_{r+1}}\wedge\mathbf{e}_{i_{r+2}}\wedge\ldots\wedge\mathbf{e}_{i_l}\right).
\]
$\cdot$ Given $1\leq j_1<\ldots<j_{l-1}\leq n$, set
\begin{align*}
& \partial_l \left(\mathbf{e}_{j_1}\wedge\ldots\wedge\mathbf{e}_{j_{l-1}}\wedge u\right):= (-1)^{l-1}\left(\Theta-(s_{j_1}\cdots s_{j_{l-1}})\right)\left(\mathbf{e}_{j_1}\wedge\ldots\wedge\mathbf{e}_{j_{l-1}}\right)\\ & +\sum_{r=1}^{l-1}(-1)^{r-1}\varphi(x_{j_r})\left(\mathbf{e}_{j_1}\wedge\ldots\wedge\mathbf{e}_{j_{r-1}}\wedge\mathbf{e}_{j_{r+1}}\wedge\mathbf{e}_{j_{r+2}}\wedge\ldots\wedge\mathbf{e}_{j_{l-1}}\wedge u\right).
\end{align*}
\end{df}

Before going on, we make the following useful:

\begin{disc}\label{unos diagramas para calentar 2}
Given a free, finitely generated left $S$-module $M$ (whence $M$ is abstractly isomorphic to $S^{\oplus r}$ for some $r\in\N$), we denote by $\xymatrix@1{S[\Theta;\varphi]\otimes_S S^{\oplus r}\ar[r]^-{\lambda_M}& S[\Theta;\varphi]^{\oplus r}}$ the natural isomorphism of left $S[\Theta;\varphi]$-modules given by the assignment $s\otimes m\longmapsto sm$; the reader will easily note that, for each $0\leq l\leq n$, we have the following commutative diagram:
\[
\xymatrix{S[\Theta;\varphi]\otimes_S K_{l+1}\ar[rr]^-{\mathbbm{1}_{S[\Theta;\varphi]}\otimes d_{l+1}}\ar[d]_-{\lambda_{K_{l+1}}}& & S[\Theta;\varphi]\otimes_S K_l\ar[d]^-{\lambda_{K_l}}\\ \bigoplus_{1\leq i_1<\ldots<i_{l+1}\leq n}S[\Theta;\varphi](\mathbf{e}_{i_1}\wedge\ldots\wedge\mathbf{e}_{i_{l+1}})\ar[rr]^-{d'_{l+1}}& 
& \bigoplus_{1\leq i_1<\ldots<i_l\leq n}S[\Theta;\varphi](\mathbf{e}_{i_1}\wedge\ldots\wedge\mathbf{e}_{i_l}).}
\]
Here, $d'_{l+1}$ denotes the map $\partial_{l+1}$ restricted to the direct summand
\[
\bigoplus_{1\leq i_1<\ldots<i_{l+1}\leq n}S[\Theta;\varphi](\mathbf{e}_{i_1}\wedge\ldots\wedge\mathbf{e}_{i_{l+1}}).
\]
As we shall see quickly, this fact turns out to be very useful in what follows.
\end{disc}

The first thing we have to check out is that $\FK_{\bullet} (x_1,\ldots ,x_n)$ defines a chain complex in the category of left $S[\Theta;\varphi]$-modules. This fact follows from the next:

\begin{prop}\label{esto es un complejo tipo Koszul por la izquierda}
For any $0\leq l\leq n$, one has that $\partial_l\partial_{l+1}=0$.
\end{prop}

\begin{proof}
Regarding the very definition of the $\partial$'s, we only have to distinguish two cases.

\vskip 2mm

$\cdot$ Given $1\leq i_1<\ldots<i_{l+1}\leq n$ one has, keeping in mind Discussion \ref{unos diagramas para calentar 2}, that
\begin{align*}
& \partial_l \left(\partial_{l+1}\left(\mathbf{e}_{i_1}\wedge\ldots\wedge\mathbf{e}_{i_{l+1}}\right)\right)=d'_l \left(d'_{l+1}\left(\mathbf{e}_{i_1}\wedge\ldots\wedge\mathbf{e}_{i_{l+1}}\right)\right)\\ & =\left(\lambda_{K_{l-1}}\circ\left(\mathbbm{1}_{S[\Theta;\varphi]}\otimes d_{l-1}\right)\circ\lambda_{K_l}^{-1}\right)\circ\left(\lambda_{K_l}\circ\left(\mathbbm{1}_{S[\Theta;\varphi]}\otimes d_l\right)\circ\lambda_{K_{l+1}}^{-1}\right)\left(\mathbf{e}_{i_1}\wedge\ldots\wedge\mathbf{e}_{i_{l+1}}\right)\\ & =\left(\lambda_{K_{l-1}}\circ\left(\mathbbm{1}_{S[\Theta;\varphi]}\otimes(d_l\circ d_{l+1})\right)\circ\lambda_{K_{l+1}}^{-1}\right)\left(\mathbf{e}_{i_1}\wedge\ldots\wedge\mathbf{e}_{i_{l+1}}\right)\\ & =\left(\lambda_{K_{l-1}}\circ\left(\mathbbm{1}_{S[\Theta;\varphi]}\otimes 0\right)\circ\lambda_{K_{l+1}}^{-1}\right)\left(\mathbf{e}_{i_1}\wedge\ldots\wedge\mathbf{e}_{i_{l+1}}\right)=0;
\end{align*}
indeed, notice that $d_l d_{l+1}=0$ because they are the usual chain differentials in the Koszul chain complex of $S$ with respect to $x_1,\ldots ,x_n$.

\vskip 2mm

$\cdot$ Given $1\leq j_1<\ldots<j_l\leq n$, one has that
\begin{align*}
& \partial_l \left(\partial_{l+1}\left(\mathbf{e}_{j_1}\wedge\ldots\wedge\mathbf{e}_{j_l}\wedge u\right)\right)=\partial_l \left((-1)^l \left(\Theta-(s_{j_1}\cdots s_{j_l})\right)\left(\mathbf{e}_{j_1}\wedge\ldots\wedge\mathbf{e}_{j_l}\right)\right. +\\ & \left.\sum_{r=1}^l (-1)^{r-1}\varphi(x_{j_r})\left(\mathbf{e}_{j_1}\wedge\ldots\wedge\mathbf{e}_{j_{r-1}}\wedge\mathbf{e}_{j_{r+1}}\wedge\mathbf{e}_{j_{r+2}}\wedge\ldots\wedge\mathbf{e}_{j_l}\wedge u\right)\right)=\\ & \sum_{r=1}^l (-1)^{r+l-1}\left(\varphi(x_{j_r})\Theta-(s_{j_1}\cdots s_{j_l})x_{j_r}\right)\left(\mathbf{e}_{j_1}\wedge\ldots\wedge\mathbf{e}_{j_{r-1}}\wedge\mathbf{e}_{j_{r+1}}\wedge\mathbf{e}_{j_{r+2}}\wedge\ldots\wedge\mathbf{e}_{j_l}\right)+\\ 
& \sum_{r=1}^l\sum_{k=1}^{r-1}
(-1)^{r+k-2}\varphi(x_{j_k}x_{j_r})\left(\mathbf{e}_{j_1}\wedge\ldots\wedge\mathbf{e}_{j_{k-1}}\wedge\mathbf{e}_{j_{k+1}}\wedge\mathbf{e}_{j_{k+2}}\wedge\ldots\wedge\mathbf{e}_{j_{r-1}}\wedge\mathbf{e}_{j_{r+1}}\wedge\mathbf{e}_{j_{r+2}}\wedge\ldots
\wedge\mathbf{e}_{j_l}\wedge u\right)+\\ & \sum_{r=1}^l\sum_{k=r}^l (-1)^{r+k-2}\varphi(x_{j_k}x_{j_r})\left(\mathbf{e}_{j_1}\wedge\ldots\wedge\mathbf{e}_{j_{r-1}}\wedge\mathbf{e}_{j_{r+1}}\wedge\mathbf{e}_{j_{r+2}}\wedge\ldots\wedge\mathbf{e}_{j_{k-1}}\wedge\mathbf{e}_{j_{k+1}}\wedge\mathbf{e}_{j_{k+2}}\wedge\ldots\wedge\mathbf{e}_{j_l}\wedge u\right)+\\ & \sum_{r=1}^l (-1)^{l+r-2}\left(\varphi(x_{j_r})\Theta-(s_{j_1}\cdots s_{j_{r-1}}s_{j_{r+1}}s_{j_{r+2}}\cdots s_{j_l})\varphi(x_{j_r})\right)\left(\mathbf{e}_{j_1}\wedge\ldots\wedge\mathbf{e}_{j_{r-1}}\wedge\mathbf{e}_{j_{r+1}}\wedge\mathbf{e}_{j_{r+2}}\wedge\ldots\wedge\mathbf{e}_{j_l}\right).
\end{align*}

\vskip 2mm
Starting from the top, the first summand (respectively, the second summand) cancels out the fourth summand (respectively, the third summand) and therefore the whole expression vanishes, just what we finally wanted to check.
\end{proof}

\begin{rk} \label{matrices}
The differentials of the $\varphi$-Koszul complex
\[
\FK_{\bullet}(x_1,\dots,x_n):=\xymatrix{0\ar[r]& \FK_{n+1}\ar[r]^-{\partial_{n+1}}& \FK_n\ar[r]^-{\partial_n}& \ldots\ar[r]^-{\partial_1}& \FK_0\ar[r]&  0.}
\]
are described as follows:

\begin{enumerate}

\item $\partial_{n+1}$ is represented by right multiplication by matrix $\begin{pmatrix} (-1)^n\left(\Theta-(s_1\cdots s_n)\right)& M_n^{[\varphi]}\end{pmatrix}$.

\item For each $1\leq l\leq n-1$, $\partial_{l+1}$ is represented by right multiplication by matrix
\[
\left(\begin{array}{c|c}
M_{l+1}& \mathbf{0}\\ \hline
(-1)^l D_l& M_l^{[\varphi]}\end{array}\right),
\]
where $D_l$ is a diagonal matrix with non-zero entries $\Theta-(s_{i_1}\cdots s_{i_l})$, $1\leq i_1<\ldots <i_l\leq n$. 

\item $\partial_1$ is represented by right multiplication by matrix $\begin{pmatrix} x_1 & \ldots & x_n & \Theta -1\end{pmatrix}^T$.

\end{enumerate}

\vskip 2mm

For example, when $n=2$, $\FK_{\bullet} (x_1,x_2)$ boils down to the chain complex
\[
\xymatrix{0\ar[r]& S[\Theta ;\varphi]\ar[r]^-{\partial_3}& S[\Theta ;\varphi]^{\oplus 3}\ar[r]^-{\partial_2}& S[\Theta ;\varphi]^{\oplus 3}\ar[r]^-{\partial_1}& S[\Theta;\varphi]\ar[r]& 0,}
\]
where $\partial_3$, $\partial_2$ and $\partial_1$ are given by right multiplication by matrices 
\[
\begin{pmatrix} \Theta-(s_1s_2) & -\varphi(x_2)& \varphi(x_1)\end{pmatrix} \hskip 5mm
\left(\begin{array}{cc|c}
-x_2& x_1 & 0\\ \hline
 s_1-\Theta & 0& \varphi(x_1)\\ 0& s_2-\Theta & \varphi(x_2) \end{array}\right) \hskip 5mm
\begin{pmatrix} x_1 \\ x_2 \\ \Theta -1 \end{pmatrix}
\]

\end{rk}

\vskip 2mm

Now, we want to single out the following technical fact because it will play a key role during the proof of the first main result of this paper (see Theorem \ref{la exactitud del Cartier Koszul}).

\begin{lm}\label{una cierta commutacion de matrices dos}
Preserving the notations introduced in Remark \ref{matrices}, one has, for any $0\leq l\leq n$, that $M_l^{[\varphi]} D_{l-1}=D_lM_l$.
\end{lm}

\begin{proof}
Fix $0\leq l\leq n$. Proposition \ref{esto es un complejo tipo Koszul por la izquierda} implies that $\partial_l \partial_{l+1}=0$. Then, according to Remark  \ref{matrices}, this equality corresponds to
\[
\left(\begin{array}{c|c}
M_{l+1}& \mathbf{0}\\ \hline
(-1)^l D_l& M_l^{[\varphi]}\end{array}\right)
\left(\begin{array}{c|c}
M_l& \mathbf{0}\\ \hline
(-1)^{l-1} D_{l-1}& M_{l-1}^{[\varphi]}\end{array}\right)=\mathbf{0}.
\]
In particular, we must have $(-1)^l D_lM_l+(-1)^{l-1}M_l^{[\varphi]}D_{l-1}=\mathbf{0}$, which is equivalent to say that
\[
(-1)^l\left(D_lM_l -M_l^{[\varphi]}D_{l-1}\right)=\mathbf{0}.
\]
Whence $M_l^{[\varphi]} D_{l-1}=D_lM_l$, just what we finally wanted to show.
\end{proof}

Before showing our main result, we want to establish a certain technical fact, which is interesting in its own right; namely:

\begin{prop}\label{un modulo plano en medio}
$S[\Theta;\varphi]$ is a flat right $S$-module.
\end{prop}

\begin{proof}
By the very definition of left skew polynomial rings,
\[
S[\Theta ;\varphi]=\bigoplus_{e\geq 0}S\Theta^e.
\]
Since a direct sum of right $S$--modules is flat if and only if it
is so all its direct summands \cite[Proposition 3.46 (ii)]{Rotman2009},
it is enough to check that, for any $e\geq 0$, $S\Theta^e$ is a flat right $S$-module.

\vskip 2mm

Fix $e\geq 0$. Firstly, albeit the notation $S\Theta^e$ might suggest that it is just a left $S$-module, this is not the case because $S\Theta^e$ can be identified with $\Theta^e \varphi_*^e S$; from this point of view, it is clear that $S\Theta^e$ may be also regarded as a right $S$-module. Therefore, keeping in mind the previous identification one has that the map $\xymatrix@1{\Theta^e \varphi_*^e S\ar[r]& \varphi_*^e S}$ given by the assignment $\Theta^e \varphi_*^e s\longmapsto\varphi_*^e s$ defines an abstract isomorphism of right $S$-modules, whence $S\Theta^e$ is (abstractly) isomorphic to $\varphi_*^e S$ in the category of right $S$-modules and then the result follows from the fact that $\varphi_*^e S$ is a flat right $S$-module because of the flatness
of $\varphi$; the proof is therefore completed.
\end{proof}

\begin{rk}\label{right flatness of Frobenius}
When $S$ is a commutative Noetherian regular local ring of prime characteristic
$p$ and $F$ is the Frobenius map on $S,$ the fact that $S[\Theta;F]$
is a flat right $S$--module was already observed by Y.\,Yoshino
\cite[Proof of Example (9.2)]{Yoshino1994}.
\end{rk}

Next result provides some useful properties of $\FK_{\bullet} (x_1,\ldots ,x_n)$.

\begin{prop}\label{unos cuantos hechos del Frobenius-Koszul}
$H_0 \left(\FK_{\bullet} (x_1,\ldots ,x_n)\right)\cong S/I_n$ as left $S[\Theta;\varphi]$-modules, where $I_n=\langle x_1,\ldots ,x_n \rangle;$ moreover, $H_{n+1} \left(\FK_{\bullet} (x_1,\ldots ,x_n)\right)=0$.
\end{prop}

\begin{proof}
First we notice that, as in the case of the Frobenius morphism studied by Y.\,Yoshino (see Proposition \ref{localizacion con Frobenius}), we have $\frac{S[\Theta ;\varphi]}{S[\Theta ;\varphi](\Theta -1)}\cong S$. Then, using Remark \ref{matrices} it follows that
\[
H_0 \left(\FK_{\bullet} (x_1,\ldots ,x_n)\right)=\frac{S[\Theta;\varphi]}{\im (\partial_1)}\cong\frac{S[\Theta ;\varphi]}{S[\Theta ;\varphi]I_n+S[\Theta ;\varphi](\Theta -1)}\cong S/I_n.
\]

Now, consider the composition
\[
\xymatrix{S[\Theta;\varphi]\ar[r]^-{\partial_{n+1}}& S[\Theta;\varphi] \oplus S[\Theta;\varphi]^{\oplus n}\ar[r]^-{\pi}& S[\Theta;\varphi]^{\oplus n},}
\]
where $\pi$ denotes the corresponding projection. 
In this way, we have that $\pi\partial_{n+1}$ turns out to be, up to isomorphisms, 
$\mathbbm{1}_{S[\Theta;\varphi]}\otimes d_n^{[\varphi]}$ 
(indeed, this fact follows directly from the commutative square established in 
Discussion \ref{unos diagramas para calentar 2}); regardless, since $d_n^{[\varphi]}$ is an injective 
homomorphism between free left $S$-modules, and $S[\Theta;\varphi]$ is a flat right $S$-module 
(cf.\,Proposition \ref{un modulo plano en medio}), one has that $\mathbbm{1}_{S[\Theta;\varphi]}\otimes d_n^{[\varphi]}$ 
is an injective homomorphism between free left $S[\Theta;\varphi]$-modules. Therefore, $\pi\partial_{n+1}$ is also an 
injective homomorphism, whence $\partial_{n+1}$ is so. This fact concludes the proof.
\end{proof}

Now, we state and prove the first main result of this paper, which is the following:

\begin{teo}\label{la exactitud del Cartier Koszul}
The $\varphi$-Koszul complex $\FK_{\bullet} (x_1,\ldots ,x_n)$ provides a free resolution of $S/I_n$ in the category of left $S[\Theta;\varphi]$-modules.
\end{teo}

\begin{proof}
By Proposition \ref{unos cuantos hechos del Frobenius-Koszul}, it is enough to check, for any $1\leq l\leq n$, that $H_l \left(\FK_{\bullet}(x_1,\ldots ,x_n)\right)=0$.

So, fix $1\leq l\leq n$. Our goal is to show that $\ker (\partial_l)\subseteq\im (\partial_{l+1})$; in other words, we have to prove that the chain complex $\xymatrix@1{\FK_{l+1}\ar[r]^-{\partial_{l+1}}& \FK_l\ar[r]^-{\partial_l}& \FK_{l-1}}$ is midterm exact. First of all, remember that $\FK_l=K'_l \oplus K''_l$, where
\begin{align*}
& K'_l:=\bigoplus_{1\leq i_1<\ldots<i_l\leq n}S[\Theta ;\varphi](\mathbf{e}_{i_1}\wedge\ldots\wedge\mathbf{e}_{i_l}),\text{ and}\\
& K''_l:=\bigoplus_{1\leq j_1<\ldots<j_{l-1}\leq n}S[\Theta ;\varphi](\mathbf{e}_{j_1}\wedge\ldots\wedge\mathbf{e}_{j_{l-1}}\wedge u).
\end{align*}
Furthermore, Discussion \ref{unos diagramas para calentar 2} implies that the chain complexes
\[
K'_{\bullet}:\ \xymatrix{0\ar[r]& K'_n\ar[r]^-{d'_n}& K'_{n-1}\ar[r]& \ldots \ar[r]& K'_2\ar[r]^-{d'_2}\ar[r]& K'_1\ar[r]^-{d'_1}& K'_0\ar[r]& 0}
\]
and
\[
(K'_{\bullet})^{[\varphi]}:\ \xymatrix{0\ar[r]& K'_n\ar[r]^-{(d'_n)^{[\varphi]}}& K'_{n-1}\ar[r]& \ldots \ar[r]& K'_2\ar[r]^-{(d'_2)^{[\varphi]}}\ar[r]& K'_1\ar[r]^-{(d'_1)^{[\varphi]}}& K'_0\ar[r]& 0}
\]
are respectively canonically isomorphic to $S[\Theta;\varphi]\otimes_S K_{\bullet}$ and $S[\Theta;\varphi]\otimes_S K_{\bullet}^{[\varphi]}$; in particular, since $S[\Theta;\varphi]$ is a flat right $S$-module (cf.\,Proposition \ref{un modulo plano en medio}), $K'_{\bullet}$ and $(K'_{\bullet})^{[\varphi]}$ are both acyclic chain complexes in the category of left $S[\Theta;\varphi]$-modules. On the other hand, we also have to keep in mind that $d'_l$ and $(d'_l)^{[\varphi]}$ are respectively represented by right multiplication by matrix $M_l$ and $M_l^{[\varphi]}$ (cf.\,Proposition \ref{unos cuantos hechos del Frobenius-Koszul} and its corresponding notation).

\vskip 2mm

Now, let $P\in\ker (\partial_l)\subseteq\FK_l$. Since $\FK_l=K'_l\oplus K''_l$, we may write $P=(P',P'')$ for certain $P'\in K'_l$ and $P''\in K''_l$; in this way, as $P\in\ker (\partial_l)$ it follows that
\[
\begin{pmatrix} \mathbf{0}& \mathbf{0}\end{pmatrix}=\begin{pmatrix} P'&  P''\end{pmatrix}\left(\begin{array}{c|c}
M_l& \mathbf{0}\\ \hline
(-1)^{l-1}D_{l-1}& M_{l-1}^{[\varphi]}\end{array}\right)=\begin{pmatrix} P'M_l +P''(-1)^{l-1}D_{l-1}& P''M_{l-1}^{[\varphi]}\end{pmatrix},
\]
which leads to the following system of equations:
\[
P'M_l +P''(-1)^{l-1}D_{l-1}=\mathbf{0},\ P''M_{l-1}^{[\varphi]}=\mathbf{0}.
\]
In particular, since $P''M_{l-1}^{[\varphi]}=\mathbf{0}$ one has that $P''\in\ker ((d'_{l-1})^{[\varphi]})=\im ((d'_l)^{[\varphi]})$; therefore, there is $Q''\in K''_l$ such that $Q''M_l^{[\varphi]} =P''$. Using this fact, it follows that
\[
P'M_l +Q''(-1)^{l-1}M_l^{[\varphi]}D_{l-1}=0.
\]
Regardless, Lemma \ref{una cierta commutacion de matrices dos} tells us that $M_l^{[\varphi]}D_{l-1}=D_l M_l$, whence
\[
P'M_l +Q''(-1)^{l-1}D_l M_l=0,
\]
which is equivalent to say that $\left(P'+Q''(-1)^{l-1}D_l\right)M_l=0$. In this way, one has that
\[
P'+Q''(-1)^{l-1}D_l\in\ker (d'_l)=\im (d'_{l+1})
\]
and therefore there exists $Q'\in K'_{l+1}$ such that $Q'M_{l+1}=P'+Q''(-1)^{l-1}D_l$.

Summing up, setting $Q:=(Q',Q'')\in\FK_{l+1}$, it follows that
\begin{align*}
\begin{pmatrix} Q'& Q''\end{pmatrix}
\left(\begin{array}{c|c}
M_{l+1}& \mathbf{0}\\ \hline
(-1)^l D_l& M_l^{[\varphi]}\end{array}\right)& =\begin{pmatrix} Q'M_{l+1}+Q''(-1)^l D_l &  Q''M_l^{[\varphi]}\end{pmatrix}\\ &=\begin{pmatrix} P'+Q''(-1)^{l-1}D_l +Q''(-1)^l D_l & P''\end{pmatrix}=\begin{pmatrix} P'& P''\end{pmatrix}
\end{align*}
and therefore we can conclude that $P\in\im (\partial_{l+1})$, which is exactly what we wanted to show.
\end{proof}
The reader will easily note that, during the proof, we have obtained the below result; we specially
thank Rishi Vyas to single out to us this fact.

\begin{prop}
There is a short exact sequence of chain complexes
\[
0\rightarrow S[\Theta;\varphi]\otimes_S K_{\bullet} (S;x_1,\ldots ,x_n)\rightarrow \FK_{\bullet} (x_1,\ldots ,x_n)\rightarrow 
\left(S[\Theta;\varphi]\otimes_S K_{\bullet} (S;\varphi(x_1),\ldots ,\varphi(x_n))\right)[1]\rightarrow 0
\]
in the category of left $S[\Theta;\varphi]$--modules.
\end{prop}

\subsection{The $\varphi$-Koszul chain complex in full generality}

Our next aim is to define the $\varphi$-Koszul chain complex over a more general setting, 
and explore some specific situations on which we can ensure that defines 
a finite free resolution.
Let $A$ be a commutative Noetherian ring containing a commutative subring $\K,$ and $y_1,\ldots ,y_n$ denote 
arbitrary elements of $A;$ in
addition, assume that we fix a flat endomorphism of $\K$--algebras
$\xymatrix@1{S:=\K[x_1,\ldots ,x_n]\ar[r]^-{\varphi}& S}$ satisfying
$\varphi (x_i)\in\langle x_i\rangle$ for any $1\leq i\leq n.$ In this
section, we regard $A$ as an $S$-algebra under $\xymatrix@1{S\ar[r]
^-{\psi}& A,}$ where $\xymatrix@1{S\ar[r]
^-{\psi}& A}$ is the natural homomorphism of $\K$-algebras which sends each $x_i$ to $y_i.$ Finally, we
suppose that there exists a $\K$--algebra homomorphism
$\xymatrix@1{A\ar[r]^-{\Phi}& A}$ making the square
\[
\xymatrix{S\ar[d]_-{\varphi}\ar[r]^-{\psi}& A\ar[d]^-{\Phi}\\
S\ar[r]_-{\psi}& A}
\]
commutative.

\begin{df}
We define the \emph{$\varphi$-Koszul chain complex} of $A$ with respect to $y_1,\ldots ,y_n$ as the chain complex $\FK_{\bullet} (y_1,\ldots ,y_n;A):=A\otimes_S \FK_{\bullet} (x_1,\ldots ,x_n).$
\end{df}

Under slightly different assumptions, we want to show that $\FK_{\bullet} (y_1,\ldots ,y_n;A)$ still
defines a finite free resolution; with this purpose in mind, first
of all we review the following notion, which was introduced independently in
\cite[Definition 2 of page 157]{Bourbakihomological} (using a different terminology) and by T.\,Kabele in \cite[Definition 2]{Kabele1971}.

\begin{df}\label{Koszul regular sequence}
Let $R$ be a commutative ring, let $s\in\N$, and let $f_1,\ldots ,f_n$ be a sequence of elements in $R$. It is said that $f_1,\ldots ,f_n$ is a \emph{Koszul regular sequence} provided the Koszul chain complex $K_{\bullet} (f_1,\ldots ,f_n;R)$ provides a free resolution of $R/I_n$, where $I_n=\langle f_1,\dots, f_n \rangle$.
\end{df}



Next statement may be regarded as a generalization of Theorem \ref{la exactitud del Cartier Koszul}; this is the main
result of this section.

\begin{teo} \label{resolution}
Let $A$ be a commutative Noetherian ring containing a commutative subring $\K,$ and let $y_1,\ldots ,y_n$ be a sequence of elements in $A.$ Moreover, we assume that $\Phi$ is flat and $y_1,\ldots ,y_n$ is an $A$-Koszul regular sequence. Then, $\FK_{\bullet} (y_1,\ldots ,y_n;A)$ defines a finite free resolution of $A/I_n$ in the category of left $A[\Theta ;\Phi]$-modules, where $I_n =\langle y_1,\ldots ,y_n \rangle.$
\end{teo}

\begin{proof}
The proof of this result is, mutatis mutandi, the same as the one of Theorem \ref{la exactitud del Cartier Koszul} replacing $S$ by $A$ and $x_1,\ldots ,x_n$ by $y_1,\ldots ,y_n$. Indeed, a simple inspection of the proof of Theorem \ref{la exactitud del Cartier Koszul} reveals that we only used there the flatness
of $\varphi$ and the fact that the Koszul chain complex $K_{\bullet} (x_1,\ldots ,x_n)$ defines a finite free resolution of $\K;$ the proof is therefore completed.
\end{proof}

\begin{rk} 
The global homological dimension of right skew polynomial rings 
was studied by K.\,L.\,Fields  in \cite{Fields1969, Fields1970}. 
An upper bound for this global dimension is $n+1$ when $\varphi$ is injective and it is 
exactly $n+1$ in the case that $\varphi$ is an automorphism. Passing to the opposite ring 
(see Remark \ref{opposite}) we would get analogous results for left skew polynomial rings. 
Notice that the length of  the $\varphi$-Koszul complex  $\FK_{\bullet} (y_1,\ldots ,y_n;A)$ is $n+1$. 
\end{rk}

\subsection{The case of the Frobenius homomorphism}
For the convenience of the reader, we will specialize the construction of the $\varphi$-Koszul complex 
to the case where $A$ is a commutative Noetherian regular ring 
of positive characteristic $p>0$, and $\varphi=F^e$ is an $e$-th iteration of the Frobenius morphism.
In this case, we will denote this complex simply as \emph{Frobenius-Koszul complex}. Namely we have:
\[
\FK_{\bullet}(x_1,\dots,x_n):=\xymatrix{0\ar[r]& \FK_{n+1}\ar[r]^-{\partial_{n+1}}& \FK_n\ar[r]^-{\partial_n}& \ldots\ar[r]^-{\partial_1}& \FK_0\ar[r]&  0.}
\]

\vskip 2mm 

\noindent where, for each $0\leq l\leq n+1$,
\begin{align*}
\FK_l:=& \bigoplus_{1\leq i_1<\ldots<i_l\leq n}A[\Theta;F^e](\mathbf{e}_{i_1}\wedge\ldots\wedge\mathbf{e}_{i_l}) \hskip 2mm \oplus\bigoplus_{1\leq j_1<\ldots<j_{l-1}\leq n}A[\Theta;\varphi](\mathbf{e}_{j_1}\wedge\ldots\wedge\mathbf{e}_{j_{l-1}}\wedge u),
\end{align*}

\noindent and the differentials  $\xymatrix@1{\FK_l\ar[r]^-{\partial_l}& \FK_{l-1}}$ are given by: 

\vskip 2mm

$\cdot$ For $1\leq i_1<\ldots<i_l\leq n$, set
\[
\partial_l \left(\mathbf{e}_{i_1}\wedge\ldots\wedge\mathbf{e}_{i_l}\right):=\sum_{r=1}^l (-1)^{r-1}x_{i_r}\left(\mathbf{e}_{i_1}\wedge\ldots\wedge\mathbf{e}_{i_{r-1}}\wedge\mathbf{e}_{i_{r+1}}\wedge\mathbf{e}_{i_{r+2}}\wedge\ldots\wedge\mathbf{e}_{i_l}\right).
\]

$\cdot$ For $1\leq j_1<\ldots<j_{l-1}\leq n$, set
\begin{align*}
& \partial_l \left(\mathbf{e}_{j_1}\wedge\ldots\wedge\mathbf{e}_{j_{l-1}}\wedge u\right):= (-1)^{l-1}\left(\Theta-(x^{p^e-1}_{j_1}\cdots x^{p^e-1}_{j_{l-1}})\right)\left(\mathbf{e}_{j_1}\wedge\ldots\wedge\mathbf{e}_{j_{l-1}}\right)\\ & +\sum_{r=1}^{l-1}(-1)^{r-1} x^{p^e}_{j_r}\left(\mathbf{e}_{j_1}\wedge\ldots\wedge\mathbf{e}_{j_{r-1}}\wedge\mathbf{e}_{j_{r+1}}\wedge\mathbf{e}_{j_{r+2}}\wedge\ldots\wedge\mathbf{e}_{j_{l-1}}\wedge u\right).
\end{align*}

\begin{rk}

For  $n=2$, $\FK_{\bullet} (x_1,x_2)$ is just
\[
\xymatrix{0\ar[r]& A[\Theta ;F^e]\ar[r]^-{\partial_3}& A[\Theta ;F^e]^{\oplus 3}\ar[r]^-{\partial_2}& A[\Theta ;F^e]^{\oplus 3}\ar[r]^-{\partial_1}& A[\Theta;F^e]\ar[r]& 0,}
\]
where $\partial_3$, $\partial_2$ and $\partial_1$ are given by right multiplication by matrices 
\[
\begin{pmatrix} \Theta-(x_1^{p^e-1}x_2^{p^e-1}) & -x_2^{p^e}& x_1^{p^e}\end{pmatrix} \hskip 5mm
\left(\begin{array}{cc|c}
-x_2& x_1 & 0\\ \hline
 x_1^{p^e-1}-\Theta & 0& x_1^{p^e}\\ 0& x_2^{p^e-1}-\Theta & x_2^{p^e} \end{array}\right) \hskip 5mm
\begin{pmatrix} x_1 \\ x_2 \\ \Theta -1 \end{pmatrix}
\]

\end{rk}

\vskip 2mm 

As a  direct consequence of
Theorem \ref{resolution} we obtain the below:

\begin{cor}\label{resolution for the Frobenius}
Let $A$ be a commutative Noetherian regular ring of prime characteristic
$p,$ and let $y_1,\ldots ,y_n$ be an $A$--Koszul regular sequence. Then, 
$\FK_{\bullet} (y_1,\ldots ,y_n;A)$ defines a finite free resolution of $A/I_n$ in the category of 
left $A[\Theta ;F^e]$-modules, where $I_n =\langle y_1,\ldots ,y_n \rangle.$
\end{cor}

\subsection*{Acknowledgements}
The second author thanks Rishi Vyas for fruitful discussions about the material presented here.

\bibliographystyle{alpha}
\bibliography{AFBoixReferences}

\end{document}